\documentclass[10pt,a4paper]{amsart}
\usepackage{amssymb}
\usepackage{amsmath}
\usepackage{amsfonts,a4wide}
\usepackage{bbm}
\usepackage{enumerate}
\usepackage[latin1]{inputenc}
\usepackage{shadow}
\usepackage{slashed}
\usepackage[all]{xy}
\usepackage{tikz-cd}
\usepackage[letterpaper, margin=4cm]{geometry}
\usepackage{cleveref}

\newcommand\NN{\mathbb{N}}
\newcommand\QQ{\mathbb{Q}}
\newcommand\CC{\mathbb{C}}
\newcommand\PP{\mathbb{P}}
\newcommand\ZZ{\mathbb{Z}}

\DeclareMathOperator {\APL}{A_{PL}}

\DeclareMathOperator {\Hom}{Hom}
\DeclareMathOperator {\Sing}{Sing}
\DeclareMathOperator {\SP}{SP}
\DeclareMathOperator {\Spec}{Spec}

\theoremstyle{definition}
\newtheorem{theorem}[subsection]{Theorem}
\newtheorem{definition}[subsection]{Definition}
\newtheorem{proposition}[subsection]{Proposition}

\newtheorem{corollary}[subsection]{Corollary}
\newtheorem{lemma}[subsection]{Lemma}

\newtheorem{remark}[subsection]{Remark}
\newtheorem{example}[subsection]{Example}

\title{Infinite symmetric products of rational algebras and spaces}
\author{Jiahao Hu}
\address{Stony Brook University, Department of Mathematics, 100 Nicolls Road, 11794 Stony Brook}
\email{jiahao.hu@stonybrook.edu}

\author{Aleksandar Milivojevi\'c}
\address{Max Planck Institute for Mathematics,
Vivatsgasse 7,
53111 Bonn}
\email{milivojevic@mpim-bonn.mpg.de}

\begin{document}
\begin{abstract}
We show that the infinite symmetric product of a connected graded-commutative algebra over $\QQ$ is naturally isomorphic to the free graded-commutative algebra on the positive degree subspace of the original algebra. In particular, the infinite symmetric product of a connected commutative (in the usual sense) graded algebra over $\QQ$ is a polynomial algebra. Applied to topology, we obtain a quick proof of the Dold--Thom theorem in rational homotopy theory for connected spaces of finite type. We also show that finite symmetric products of certain simple free graded commutative algebras are free; this allows us to determine minimal Sullivan models for finite symmetric products of complex projective spaces.
\end{abstract}

\maketitle

\section{Introduction}

The construction of symmetric products naturally appears in both algebra and topology. In algebra, given a commutative ring $R$, its $n$--fold symmetric product is the invariant subring $(R^{\otimes n})^{\Sigma_n}$ of its $n$--fold tensor power under the natural action of the $n^\text{th}$ symmetric group $\Sigma_n$ by permuting the factors. The infinite symmetric product is defined as the inverse limit of these objects, in an appropriate sense. The fundamental theorem of symmetric polynomials states that the $n$--fold symmetric product (including the case of taking $n$ to infinity) of the single-variable polynomial ring $\ZZ[x]$ is a free polynomial ring generated by the elementary symmetric functions. In contrast, the $n$--fold symmetric product of the polynomial ring $\ZZ[x_1,\dots,x_m]$ in more than one variable is not smooth after tensoring with $\CC$, let alone free. Geometrically speaking, $(R^{\otimes n})^{\Sigma_n}$ is the ring of regular functions on the affine variety $\Spec(R)^{\times n}/\Sigma_n$ obtained by quotienting the $n$--fold product of $\Spec(R)$ by the natural permutation action of $\Sigma_n$. Consequently the $n$--fold symmetric product of an $m$-dimensional smooth complex algebraic variety is smooth if and only if at least one of $m,n$ is one. Nevertheless, should one tensor $\ZZ[x_1,\dots,x_m]$ with $\QQ$ and let $n$ pass to infinity, then surprisingly the singularities are resolved \cite{V05}, \cite{D99}: the \textit{infinite} symmetric product of $\QQ[x_1,\dots,x_m]$ is a free polynomial algebra.

On the topological side, for any pointed connected topological space $X$, one can form its $n$--fold symmetric product $X^{\times n}/\Sigma_n$, as well as its infinite symmetric product which is the direct limit of these spaces. Through the latter, Dold and Thom revealed a stunning relation between homology and homotopy in the 1950's \cite{DT58}: the homotopy groups of the \textit{infinite} symmetric product of a connected topological space are naturally isomorphic to the reduced homology groups of the original space.

In the two decades to follow, a bridge between algebra and topology was built by Quillen and Sullivan giving us that the rational homotopy theoretic information of a space is encoded in a corresponding differential graded (Lie) algebra over the rationals.

It is our motivation to further understand the relation between the algebraic and topological sides of the (infinite) symmetric product construction through the lens of rational homotopy theory; see \cite{FT10} for work on the rational homotopy types of finite symmetric products, which led the authors to the present note. We give a general structural theorem for the infinite symmetric product of a connected \textit{graded-commutative} algebra over the rationals (Section 2); in particular, we show this algebra is free. In contrast, finite symmetric products are generally not free. However, we show that the finite symmetric products of a free graded-commutative algebra on two generators, one in even degree and one in odd degree, are free. We then give a quick conceptual proof of the Dold--Thom theorem in rational homotopy theory for spaces of finite type (Section 3). Further, we compute minimal Sullivan models (in particular obtaining the rational homotopy groups) for finite symmetric products of complex projective spaces.

\subsection*{Acknowledgements} We thank Yves F\'elix and Daniel Tanr\'e for pointing out to us the results of \cite{ScSt99}. Proposition 3 therein largely overlaps with our \Cref{thm:topology}, but the proofs differ, as Scheerer--Stelzer make use of the Dold--Thom theorem to reach their result; meanwhile, our \Cref{thm:topology}, from which we recover the Dold--Thom theorem, is derived from the algebraic considerations in Section 2 herein. We further thank Steven Sam and Andrew Snowden for helpful comments and suggestions.

\section{Symmetric products of rational algebras}
Henceforth, unless otherwise noted, by an algebra we will always mean a \textit{graded-commutative} algebra over $\QQ$ concentrated in non-negative degrees, denoted by $A=\oplus_{d\ge 0} A_d$. If $a,b$ are two homogeneous elements, then $ab=(-1)^{\deg a \cdot\deg b}ba$. We will denote the subspace of positive degree (resp. degree $\le n$) elements by $A_+$ (resp. $A_{\le n}$).
We say $A$ is \textit{connected} if $A_0$ is generated by $1\in A$ (i.e. $A_0\cong \QQ$). A connected algebra has a canonical augmentation $\varepsilon: A\to \QQ$ by projection onto $A_0$. 
We say $A$ is of \textit{finite type} if $\beta_d:=\dim A_d$ is finite for all $d$.

For $V$ a graded vector space, by $\Lambda(V)$ we denote the free algebra on $V$. For notational convenience, we will sometimes replace $V$ in the notation $\Lambda(V)$ by a homogeneous basis for $V$. For instance, $\Lambda(x)$ denotes the free algebra on a one-dimensional graded vector space, and it is a polynomial algebra if $\deg x$ is even, and an exterior algebra if $\deg x$ is odd. If $V\cong U\oplus W$, then $\Lambda(V)\cong \Lambda(U)\otimes\Lambda(W)$.
\begin{definition} Let $A$ be a connected algebra. We define its $n$--fold symmetric product $\SP^n(A)$ to be $(A^{\otimes n})^{\Sigma_n}$, i.e. the subalgebra of $A^{\otimes n}$ consisting of elements invariant under the action of the $n^\text{th}$ symmetric group $\Sigma_n$ by permuting the factors with the usual Koszul signs (for example, the permutation $(12)$ takes $a\otimes b$ to $(-1)^{\deg a \cdot \deg b} b \otimes a$). The projection
\[
A^{\otimes {n+1}}\to A^{\otimes n}, \quad a_1\otimes \dots\otimes a_n\otimes a_{n+1}\mapsto \varepsilon(a_{n+1}) a_1\otimes \dots \otimes a_n
\]
induces a morphism  $\pi_n^{n+1}: \SP^{n+1}(A)\to \SP^{n}(A)$. We define the infinite symmetric product $\SP(A)$ of $A$ to be the inverse limit of the system $\{\pi_n^{n+1}\}_{n\in\NN}$, that is $\SP(A):=\lim\limits_{\leftarrow n} \SP^n(A)$. For each $m>n$, we denote $\pi_n^m:=\pi_{n}^{n+1}\circ\cdots\circ\pi_{m-1}^{m}: \SP^m(A)\to \SP^n(A)$ and $\pi_n:=\lim\limits_{\leftarrow m} \pi_n^m: \SP(A)\to \SP^n(A)$. It is clear that $\pi_n=\pi_n^{m}\circ\pi_m$.
\end{definition}

We emphasize that this inverse limit is taken in the category of graded-commutative algebras. As such, it can be constructed as follows: first we take the degree-wise inverse limits $\lim\limits_{\leftarrow n} \SP^n(A)_d$, and then we form the direct sum $\bigoplus_d \lim\limits_{\leftarrow n} \SP^n(A)_d$ with its inherited algebra structure.

\begin{remark} Consider the truncated polynomial algebras $\QQ[x]/(x^n)$, which  form an inverse system by quotienting. In the category of commutative algebras, the inverse limit of this system is the formal power series algebra $\QQ[[x]]$. Likewise, if we set $\deg x = 0$, the inverse limit of this system in the category of graded-commutative algebras will be $\QQ[[x]]$. However, if we set $\deg x = 2$, thus placing our system in the category of connected graded-commutative algebras, the inverse limit will be the polynomial algebra $\QQ[x]$. \end{remark}

Now let $A$ be a fixed connected algebra. We consider the algebra map $$\phi_n: \Lambda(A_+)\to \SP^n(A)$$ defined on elements of $A_+$ by $$a\mapsto [a]=a\otimes 1 \otimes \cdots \otimes 1 + \cdots + 1 \otimes \cdots \otimes 1 \otimes a.$$
It is clear from construction that $\phi_n=\pi_n^{n+1}\circ\phi_{n+1}$. We set $$\phi:=\lim\limits_{\leftarrow n}\phi_n:\Lambda(A_+)\to \SP(A).$$

\begin{lemma}\label{lem:onto}
$\phi_n: \Lambda(A_+)\to \SP^n(A)$ is onto.
\end{lemma}

\begin{proof}
For simplicity of computation we will use the isomorphism of algebras $\Lambda^n A \cong \SP^n(A)$ given by \cite[\textsection 2]{FT10} $$a_1 \wedge \cdots \wedge a_n \mapsto \sum_{\sigma \in \Sigma_n} \pm a_{\sigma(1)} \otimes \cdots \otimes a_{\sigma(n)},$$ where the multiplication on $\Lambda^n A$ is determined by $$(a_1\wedge \cdots \wedge a_n) \ast (b_1 \wedge \cdots \wedge b_n) = \sum_{\sigma \in \Sigma_n} \pm (a_1 b_{\sigma(1)}) \wedge \cdots \wedge (a_n b_{\sigma(n)}).$$ For example, we have $(a\wedge 1)\ast (b \wedge 1) = a\wedge b + (ab) \wedge 1$. Under this identification, $\phi_n$ maps $a$ to $\tfrac{1}{(n-1)!} a\wedge 1 \wedge \cdots \wedge 1$. Inductively we see that \begin{align*} a_1 \wedge \cdots \wedge a_{n-k} \wedge 1 \wedge \cdots \wedge 1 &= \tfrac{1}{k+1} (a_1 \wedge \cdots a_{n-k-1} \wedge 1 \wedge \cdots \wedge 1) \ast (a_{n-k} \wedge 1 \wedge \cdots \wedge 1) \\ &+ (\textrm{elements of the form } (a_1' \wedge \cdots \wedge a_{n-k-1}' \wedge 1 \wedge \cdots \wedge 1)),\end{align*} which gives us surjectivity. \end{proof}

\begin{lemma}\label{lem:iso}
	$\phi_n$ maps $\Lambda(A_+)_{\le n}$ isomorphically onto $\SP^n(A)_{\le n}$.
\end{lemma}
\begin{proof}
Assume first that $A$ is of finite type. Then by \Cref{lem:onto} it suffices to show $\SP^n(A)_{\le n}$ and $\Lambda(A_+)_{\le n}$ have the same dimension. To do so, we examine the Poincar\'e series of $\SP^n(A)$ and  $\Lambda(A_+)$. Recall the Poincar\'e series of a graded algebra of finite type is a formal power series in a formal variable $z$ whose coefficient of $z^d$ is the dimension $\beta_d$ of its degree $d$ subspace.

By \cite{M62} the Poincar\'e series for $\SP^n(A)$ is the coefficient of $t^n$ in
$$H(z,t)=\prod_{i=0}^\infty\frac{(1+z^{2i+1} t)^{\beta_{2i+1}}}{(1-z^{2i}t)^{\beta_{2i}}}.$$ Since $A$ is connected, $\beta_0=1$. Let $$G(z,t)=(1-t)H(z,t)=\prod_{i=1}^\infty\frac{(1+z^{2i-1}t)^{\beta_{2i-1}}}{(1-z^{2i}t)^{\beta_{2i}}}=\sum_{i=0}^\infty G_i(z)t^i.$$ Then the Poincar\'e series for $\SP^n(A)$ is
$$\sum_{i=0}^n G_i(z).$$ Notice that in each factor of $G$, $t$ is always multiplied with a positive power of $z$, so $G_i(z)$ is divisible by $z^i$. 

On the other hand, we claim the Poincar\'e series of $\Lambda(A_+)$ is $$\prod_{i=1}^\infty\frac{(1+z^{2i-1})^{\beta_{2i-1}}}{(1-z^{2i})^{\beta_{2i}}}=G(z,1).$$ Indeed, since the free algebra on a direct sum of vector spaces is the tensor product of the free algebras on each direct summand, and the Poincar\'e series of a tensor product is the product of the Poincar\'e series of each factor, we are reduced to considering the Poincar\'e series of the free algebra on a one-dimensional graded vector space concentrated, say, in degree $d$. If $d$ is even, then this free algebra is a polynomial algebra whose Poincar\'e series is $(1-z^d)^{-1}$. If $d$ is odd, then it is an exterior algebra whose Poincar\'e series is $1+z^d$.

Hence the dimensions of $\SP^n(A)_{\le n}$ and $\Lambda(A_+)_{\le n}$ are the same, namely $\sum_{i=0}^n G_i(z)$ reduced modulo $z^{n+1}$ and then evaluated at $z=1$. This completes the proof for finite type algebras. The general case follows from the lemma below and from the fact that taking direct limits is exact.
\end{proof}
\begin{remark}
We can define the $n$--fold symmetric product of the augmentation ideal $A_+$ by equipping the permutation invariants of the $n$--fold tensor power of the underlying graded vector space of $A_+$ with the induced multiplication. Since $\phi_{n-1}$ factors over $\pi_{n-1}^n$, it follows from \Cref{lem:onto} that
$\pi_{n-1}^n:\SP^n(A)\to \SP^{n-1}(A)$ is surjective. The kernel of this map is precisely $\SP^n(A_+)$, whose Poincar\'e series is $G_n(z)$. As graded vector spaces, we have $\SP^n(A)\cong \oplus_{i=0}^n \SP^i(A_+)$.
\end{remark}

\begin{lemma}\label{lem:redtofintp}
Let $A$ be a connected algebra and $I=\{B\,|\, B\subseteq A \text { connected of finite type}\}$ be the direct system formed by all connected finite type subalgebras of $A$ and the corresponding inclusions. We have the following induced isomorphisms:
\begin{enumerate}[(1)]
    \item $\lim\nolimits_I \Lambda(B_+)\cong\Lambda(A_+)$,
    \item $\lim\nolimits_I\SP^n(B)\cong\SP^n(A)$.
\end{enumerate}
\end{lemma}
\begin{proof}
\begin{enumerate}[(1)]
    \item The inclusion $B_+\hookrightarrow A_+$ clearly induces an injection $\Lambda(B_+)\hookrightarrow \Lambda(A_+)$. Since taking direct limits is exact, we have $\lim\nolimits_I \Lambda(B_+)\hookrightarrow\Lambda(A_+)$. To prove this map is also surjective, it suffices to show that $A_+$ is in the image. For each $a\in A_+$, $a$ is contained in some finite type subalgebra $B$ (for instance, the subalgebra generated by $a$), hence $a$ is in the image of $\Lambda(B_+)\to \Lambda(A_+)$ and therefore in the image of $\lim\nolimits_I \Lambda(B_+)\to\Lambda(A_+)$.
    \item Similarly, the inclusion $B\hookrightarrow A$ induces an injection $B^{\otimes n}\hookrightarrow A^{\otimes n}$. Then from the commutative diagram
    \[
    \begin{tikzcd}
    \SP^n(B)\arrow{r}\arrow[hookrightarrow]{d} & \SP^n(A)\arrow[hookrightarrow]{d}\\
    B^{\otimes n}\arrow[hookrightarrow]{r} & A^{\otimes n}
    \end{tikzcd}
    \]
    we see that the map $\SP^n(B)\to \SP^n(A)$ is injective. Therefore $\lim\nolimits_I\SP^n(B)\hookrightarrow\SP^n(A)$. On the other hand, each element in $\SP^n(A)$ is a symmetrization of some element $\alpha=\sum_{j=1}^k c_j\cdot a_1^j\otimes\cdots\otimes a_n^j\in A^{\otimes n}$, where $c_j\in \QQ$ and $a_i^j\in A$. All of these $a_i^j$'s are contained in some finite type subalgebra $B$ (for instance, the subalgebra generated by these $a_i^j$'s) so the symmetrization of $\alpha$ is in fact contained in $\SP^n(B)$. This proves $\lim\nolimits_I\SP^n(B)\to\SP^n(A)$ is onto, and therefore an isomorphism.
\end{enumerate}
\end{proof}

\begin{theorem}\label{thm:algebra}
Let $A$ be a connected (graded-commutative) algebra over $\QQ$. Then $\phi: \Lambda(A_+)\to\SP(A)$ is an isomorphism.
\end{theorem}

\begin{proof}
It suffices to show that $\phi$ maps $\Lambda(A_+)_{\le n}$ isomorphically onto $\SP(A)_{\le n}$ for all $n$. For each $m>n$, we have the commutative diagram:

$$
\begin{tikzcd}
                                                & \Lambda(A_+)_{\leq n} \arrow[d, "\phi"] \arrow[ldd, "\phi_m"', bend right] \arrow[rdd, "\phi_n", bend left] &                           \\
                                                & \mathrm{SP}(A)_{\leq n} \arrow[ld, "\pi_m"'] \arrow[rd, "\pi_n"]                                            &                           \\
\mathrm{SP}^m(A)_{\leq n} \arrow[rr, "\pi_n^m"] &                                                                                                             & \mathrm{SP}^n(A)_{\leq n}
\end{tikzcd}
$$
	By \Cref{lem:iso}, the maps $\phi_m$ and $\phi_n$ are isomorphisms, and so $\pi_n^m$ is also an isomorphism. It follows that $\pi_n=\lim\limits_{\leftarrow m}\pi_n^m$ is an isomorphism. Then since $\pi_n$ and $\phi_n$ are isomorphisms, we see that $\phi$ is an isomorphism.
\end{proof}

\begin{remark}
\begin{enumerate} \item As a consequence, we have that every (connected) free graded-commutative algebra $\Lambda(V)$ is the infinite symmetric product of some algebra. Indeed, consider $\QQ \oplus V$ with $\QQ$ in degree zero, with product given by $(c+v)\cdot (d+w) = cd+(cw+dv)$ for $c,d \in \QQ$, $v,w \in V$ (i.e. all products in $V$ are trivial) and units given by $c+0$. Then $\SP(\QQ\oplus V) \cong \Lambda(V)$.   \item Notice that \Cref{thm:algebra} holds over any field of characteristic zero, e.g. it applies to the $\ell$--adic cohomology of a smooth projective variety over a finite field of characteristic different from $\ell$.  \end{enumerate}
\end{remark}

Every commutative graded algebra (i.e., a commutative algebra in the ordinary sense whose multiplication respects the grading) can be viewed as a graded-commutative algebra by doubling the grading. This way, the category of commutative graded algebras over $\QQ$ embeds into the category of graded-commutative algebras over $\QQ$ as the full subcategory consisting of algebras concentrated in even degrees. Therefore, \Cref{thm:algebra} applies to connected commutative graded algebras as well:

\begin{corollary}\label{cor:gradedring}
For $R$ a connected commutative graded algebra over $\QQ$, we have $\SP(R)\cong\QQ[R_+]$. \qed
\end{corollary}

\begin{example}\label{ex:algebrasphere} Consider the commutative graded algebra $R = \QQ[x]/(x^2)$ with $\deg(x) > 0$. Then $R_+$ is one-dimensional, spanned by the element $x$. By \Cref{cor:gradedring}, $\SP(R) \cong \QQ[x]$. See also the examples at the end of Section 3, in particular the end of \Cref{ex:cpm}.
\end{example}

We apply \Cref{cor:gradedring} to the theory of multisymmetric functions. Recall that the theory of symmetric and multisymmetric functions concerns the symmetric products of the commutative graded rings $\ZZ[x]$ and $\ZZ[x_1,\dots,x_m]$ $(m>1)$ where generators are put in degree one \cite{D99}, \cite{V05}. As mentioned before, the fundamental theorem of symmetric functions asserts that $\SP^n(\ZZ[x])$ is a (free) polynomial ring. Meanwhile $\SP^n(\ZZ[x_1,\dots,x_m])$ is not free. However, \Cref{cor:gradedring} implies the following generalization of the fundamental theorem of symmetric functions over the rationals:

\begin{corollary}(cf. \cite[Corollary 3.5]{V05}, \cite{D99}) 
The algebra of multisymmetric functions over $\QQ$, i.e. $\SP(\QQ[x_1,\dots,x_m])$, is a polynomial algebra. \qed
\end{corollary}

We end this section with another generalization of the fundamental theorem of symmetric functions in the category of graded-commutative algebras over the rationals.

\begin{proposition}\label{prop:twovarfree}
Let $A=\Lambda(x,y)$ where $\deg (x)=2r$ is even and $\deg (y)=2s-1$ is odd. Then $\SP^n(A)$ is a free algebra generated by $$[x],[x^2],\dots, [x^n]; [y],[xy],\dots, [x^{n-1}y].$$
\end{proposition}
\begin{proof}
First we show that $\SP^n(A)$ is generated by $\{[x^k],[x^{k-1}y]\}_{1\le k\le n}$. From \Cref{lem:onto} we know that $\{[x^k],[x^{k-1}y]\}_{k\ge 1}$ generate $\SP^n(A)$. It remains to show that the elements $\{[x^{k}],[x^{k-1}y]\}_{k>n}$ are contained in the subalgebra generated by $\{[x^k],[x^{k-1}y]\}_{1\le k\le n}$.

Let $x_i$ denote $1\otimes\cdots \otimes x\otimes\cdots 1$, where $x$ is in the $i^{\text{th}}$ factor, and define $y_i$ in a similar fashion. Then recall by construction
\begin{align*}
    [x^k]&= x^k\otimes 1\otimes\cdots\otimes 1+\cdots+1\otimes\cdots\otimes 1\otimes x^k\\
    &=(x\otimes 1\otimes\cdots\otimes 1)^k+\cdots +(1\otimes\cdots\otimes 1\otimes x)^k =x_1^k+\cdots+x_n^k.
\end{align*}
Similarly $[x^{k-1}y]=x_1^{k-1}y_1+\cdots+x_n^{k-1}y_n$.
Let $$e_j(x)=\sum_{i_1<i_2<\dots<i_j} x_{i_1}x_{i_2}\dots x_{i_j} \quad(1\le j\le n)$$ be the elementary symmetric polynomials in $x_1,\dots,x_n$. It is well-known (from Newton's identities) that $e_j(x)$ is contained in the subalgebra generated by $[x],[x^2],\dots,[x^n]$. Notice that
\[
\prod_{i=1}^n(\lambda-x_i)=\lambda^n+\sum_{j=1}^n (-1)^j e_j(x) \lambda^{n-j}.
\]
Therefore for $1\le i\le n$ we have
\[
x_i^n+\sum_{j=1}^n (-1)^j e_j(x) x_i^{n-j}=0.
\]
Then for $k>n$, multiplying the above equation by $x_i^{k-n}$ (resp. $x_i^{k-n-1}y_i$) and summing over $i=1,\dots,n$ yields

\begin{align*}
    [x^k]&+\sum_{j=1}^n (-1)^j e_j(x) [x^{k-j}]=0,\\
    [x^{k-1}y]&+\sum_{j=1}^n (-1)^j e_j(x) [x^{k-1-j}y]=0.
\end{align*}
Hence inductively we see that $\{[x^k],[x^{k-1}y]\}_{k>n}$ are contained in the subalgebra generated by $\{[x^k],[x^{k-1}y]\}_{1\le k\le n}$.

Now, to prove that $\SP^n(A)$ is freely generated by $\{[x^k],[x^{k-1}y]\}_{1\le k\le n}$, it suffices to show the Poincar\'e series of $\SP^n(A)$ is that of the free algebra
\[
\Lambda([x],\dots,[x^n],[y],\dots,[x^{n-1}y]).
\]
Since the Poincar\'e series of $A$ is
\[
\frac{1+z^{2s-1}}{1-z^{2r}}=\sum_{i=0}^\infty z^{2ir}+\sum_{i=0}^\infty z^{2ir+2s-1},
\]
by \cite{M62} the Poincar\'e series of $\SP^n(A)$ is the coefficient of $t^n$ in
\[
\prod_{i=0}^\infty \frac{1+z^{2ir+2s-1}t}{1-z^{2ir}t}.
\]
From \cite[pp.19 Example 5]{M79} for variables $u,v,q$ we have
\[
\prod_{i=0}^\infty\frac{1-vq^i t}{1-u q^i t}=1+\sum_{n=1}^\infty \left(\prod_{i=1}^n\frac{u-vq^{i-1}}{1-q^i}\right)t^n.
\]
So, by taking $u=1, v=-z^{2s-1}, q=z^{2r}$, we see that the Poincar\'e series of $\SP^n(A)$ is
\[
\prod_{i=1}^n\frac{1+z^{2ir+2s-3}}{1-z^{2ir}}
\]
as desired.
\end{proof}

\section{Symmetric products of spaces and the rational Dold--Thom theorem}

Given a (pointed) connected cell complex $X$, one can form the $n$--fold symmetric product $\SP^n(X)$ by quotienting the $n$--fold product $X^{\times n}$ by the natural action of the symmetric group $\Sigma_n$ on $n$ elements. The space $\SP^n(X)$ naturally includes into $\SP^{n+1}(X)$, and in the limit of this directed system over $n$ one obtains the \textit{infinite symmetric product} $\SP(X)$.

To study the rational homotopy type of a space $X$, one can associate to it its commutative differential graded algebra (cdga) of piecewise-linear forms $\APL(X)$. By a \textit{model} for a space $X$ we mean any cdga which is quasi-isomorphic to $\APL(X)$. The rational cohomology of $X$ is that of any of its models, and if $X$ is furthermore a nilpotent space of \textit{finite type}, i.e. $H_n(X;\QQ)$ is finite dimensional for all $n$, then its rational (co)homotopy groups are obtained from any \textit{free} model (i.e. one whose underlying graded-commutative algebra is free) by considering the cohomology of the differential restricted to the subspace of indecomposable elements. We refer the reader to \cite{BG76}, \cite{S77}. 

We remark that $\SP(X)$ has a natural $H$-space structure, and as such its fundamental group is abelian and acts trivially on the higher homotopy groups; in particular, $\SP(X)$ is a nilpotent space. If $X$ is of finite type, then so is $\SP(X)$, as will be evident from the proof of the following theorem combined with the observation that each $\SP^n(X)$ is of finite type as well (e.g. we can apply Macdonald's formula for the Poincar\'e polynomials of finite symmetric products \cite{M62}). 

\begin{theorem}\label{thm:topology} Let $(A, d)$ be a connected model of a connected pointed space $X$. Then a model for $\SP(X)$ is given by $(\Lambda (A_+), D)$, where $D$ is obtained by extending $d$ on $A_+$ to $\Lambda (A_+)$ as a derivation.
\end{theorem}

This result, under the additional assumptions that $X$ is simply connected and that $A$ is a free algebra, has already been obtained in \cite[Proposition 3]{ScSt99}. We remark that the proof therein makes use of the Dold--Thom theorem. 

\begin{proof}
A model for the symmetric product $\SP^n(X)$ is given by $\SP^n(A) = (A^{\otimes n})^{\Sigma_n}$ equipped with the induced differential \cite[\textsection 2]{FT10}. Note that the map $\Lambda(A_+) \xrightarrow{\phi_n} \SP^n(A)$ considered in Section 2 commutes with the differentials. Namely, for $a \in A_+$ we have $Da = da \in A_+$ and hence $\phi_n(Da) = da\otimes 1 \otimes \cdots \otimes 1 + \cdots + 1 \otimes \cdots \otimes 1 \otimes da = d\phi(a)$; since $D$ on $\Lambda(A_+)$ is obtained by extending $d$ as a derivation, the claim follows.

Further, note that the composition of $\phi_n$ and the map $\SP^n(A) \to \SP^{n-1}(A)$ induced by the natural inclusion $\SP^{n-1}(X) \hookrightarrow \SP^n(X)$ is the map $\phi_{n-1}$ (see e.g. \cite[Remark 2.7]{FT10}). We thus have the commutative diagram
\vspace{1em}
$$
\begin{tikzcd}
                 &                                                  & {(\Lambda(A_+), D)} \arrow[ld, "\phi_{n+1}"'] \arrow[d, "\phi_n"] \arrow[rd, "\phi_{n-1}"] &                                                  &        \\
\cdots \arrow[r] & {(\SP^{n+1}(A),d)} \arrow[r] & {(\SP^n(A),d)} \arrow[r]                                             & {(\SP^{n-1}(A),d)} \arrow[r] & \cdots
\end{tikzcd}$$
\vspace{1em}

As a consequence, we have an induced map $(\Lambda(A_+), D) \xrightarrow{\phi} \lim\limits_{\leftarrow n} (\SP^n(A), d)$ to the inverse limit of the bottom row (in the category of cdga's). By \Cref{lem:iso}, the map $(\SP^n(A),d) \to (\SP^{n-1}(A),d)$ is an isomorphism in degrees $\leq n-1$. (In particular, the inclusion $\SP^{n-1}(X) \hookrightarrow \SP^n(X)$ is an isomorphism on rational cohomology in degrees $\leq n-2$.) Hence from our description of (the algebra underlying) the inverse limit in Section 2, we conclude that $\phi$ is an isomorphism of cdga's.

Now let us confirm that the inverse limit $\lim\limits_{\leftarrow n} (\SP^n(A), d)$ is in fact a model for $\SP(X)$. First of all, recall that the functor of piecewise-linear forms $\mathrm{Top} \xrightarrow{\APL} \mathrm{CDGA}_{\QQ}$ is given by $\APL(X) = \Hom_{\mathrm{sSet}}(\Sing_*(X), \mathfrak{A}_{\bullet}^*)$, where $\mathfrak{A}_{\bullet}^*$ denotes the simplicial cdga of polynomial forms (see e.g. \cite{Hess}). Using \cite[Theorem 14.3]{Hi14} and \cite[p.82]{BG76} (upon approximating $X$ by a cell complex if necessary), we have \begin{align*} \APL(\SP(X)) &= \APL(\lim\limits_{\rightarrow n} \SP^n(X)) = \Hom_{\mathrm{sSet}}(\Sing_*(\lim\limits_{\rightarrow n} \SP^n(X)), \mathfrak{A}_{\bullet}^*) \\ &\simeq \Hom_{\mathrm{sSet}}(\lim\limits_{\rightarrow n} \Sing_* \SP^n(X), \mathfrak{A}_{\bullet}^*) \simeq \lim\limits_{\leftarrow n} \Hom_{\mathrm{sSet}}(\Sing_* \SP^n(X), \mathfrak{A}_{\bullet}^*) \\ &= \lim\limits_{\leftarrow n} \APL(\SP^n(X)). \end{align*}

Now from the commutative diagram of models

$$ 
\begin{tikzcd}
\cdots \arrow[r] & (\SP^n(A),d) \arrow[d] \arrow[r] & (\SP^{n-1}(A),d) \arrow[d] \arrow[r] & \cdots \\ \cdots \arrow[r]  & \APL(\SP^n(X)) \arrow[r] & \APL(\SP^{n-1}(X)) \arrow[r] & \cdots
\end{tikzcd} 
$$ 
induced by the inclusions of symmetric products, where the vertical arrows are quasi-isomorphisms, we obtain a map of inverse limits $\lim\limits_{\leftarrow n} (\SP^n(A), d) \to \APL(\SP(X))$. Since the inclusion $\SP^{n-1}(X) \hookrightarrow \SP^n(X)$ is an isomorphism on rational cohomology in degrees $\leq n-2$ as we saw, it follows that this map of inverse limits is a quasi-isomorphism. We conclude that $\lim\limits_{\leftarrow n} (\SP^n(A), d)$ maps quasi-isomorphically to $\APL(\SP(X))$. \end{proof}

\begin{corollary}[Rational Dold--Thom] For a connected space $X$ of finite type, we have $\pi_*(\SP(X)) \otimes \QQ \cong \tilde{H}_*(X;\QQ)$. \end{corollary}

\begin{proof} The rational (co)homotopy groups of $\SP(X)$ are given by the linearized cohomology of a free model for $\SP(X)$. Concretely, by \Cref{thm:topology} we take our model $(\Lambda(A_+), D)$ for $\SP(X)$, where $(A, d)$ is a model for $X$, and consider the cohomology of the linear part of the differential $D$. By construction this is the same as the cohomology of $(A_+, d)$, i.e. the reduced cohomology of $X$.\end{proof}

One can also see with this line of reasoning that $\SP(X)$ has the rational homotopy type of $\prod_{i} K(\tilde{H}_i(X;\QQ),i)$. Indeed, consider the quotient map of complexes $\Lambda(A_+) \to \Lambda(A_+)/\Lambda(A_+)_+ \cdot \Lambda(A_+)_+ \cong A_+$ to the indecomposables. This is the dual of the rational Hurewicz map, and is evidently surjective. It is well known that a space has the homotopy type of a product of Eilenberg--MacLane spaces if and only if the Hurewicz map is a split injection; the claim now follows by rationalizing and dualizing.

\begin{remark} One can also use the rational Dold--Thom theorem to prove the algebraic \Cref{thm:algebra} in the finite type case; see also \cite[Proposition 3]{ScSt99}. Indeed, take a finite type connected algebra $A$. Equipping it with the trivial differential, we can realize this cdga as the cohomology of a (formal) rational space $X$. From Dold--Thom, we know that $\SP(X) \simeq \prod_i K(\widetilde{H}^i(X;\QQ), i)$, so the (rational) cohomology of $\SP(X)$ is a free algebra on the reduced cohomology of $X$, which is $A_+$. We have $\SP(A) = \SP(H^*(X))= \lim\limits_{\leftarrow n}\SP^n(H^*(X))$, and using the finite type assumption and that the cohomology of $\SP^n(X)$ stablizes, this is isomorphic to $\lim\limits_{\leftarrow n}H^*(\SP^n(X)) \cong H^*(\lim\limits_{\rightarrow n}\SP^n(X)) = H^*(\SP(X))$. We conclude that $\SP(A) \cong \Lambda(A_+)$.\end{remark}

\begin{example}\label{exm:sphere} Take the model $\left(\Lambda(x, y), d\right)$ for the two--sphere $S^2$, where $\deg x = 2, \deg y = 3$, and $dx = 0, dy = x^2$. By \Cref{thm:topology}, a model for $\SP(S^2)$ is given by $\Lambda(\{[x^k], [x^ky] \}, D)$, where $k \geq 0$, and the differential $D$ is given by \begin{align*} D[x^k] &= 0, \\ D[x^ky] &= [x^{k+2}].  \end{align*}
From here we see that the cohomology of the linear part of $D$, restricted to the indecomposables $[x^k]$, $[x^k y]$, is spanned by $[x]$. That is, $\pi_2\left(\SP(S^2)\right) \otimes \QQ \cong \QQ$ and all other rational homotopy groups are trivial. This is as expected, since $\SP(S^2)$ is homeomorphic to $\CC\PP^\infty$.

Alternatively, since $S^2$ is formal, we can take $\left(\QQ[x]/(x^2), d=0\right)$ with $\deg x = 2$ as a model for $S^2$, and immediately reach the above conclusion; see \Cref{ex:algebrasphere}. \end{example}

\begin{example}\label{ex:cpm}
Take the model $(A,d)=\left(\Lambda(x,y),dx=0, dy=x^{m+1}\right)$ for $\CC\PP^m$, where $\deg x=2$ and $\deg y=2m+1$. By \Cref{prop:twovarfree} $$\SP^n(A)=\Lambda\left([x],[x^2],\dots, [x^n]; [y],[xy],\dots, [x^{n-1}y]\right)$$ with the induced differential $d[x^k]=0,d[x^{k-1} y]=[x^{m+k}]$ for $k\ge 1$ is a \textit{free} model for $\SP^n(\CC\PP^m)$. Then by quotienting out contractible pairs of generators, we obtain a minimal Sullivan model for $\SP^n(\CC\PP^m)$, that is
\begin{align*}
    \Lambda\left([x],\dots,[x^n]; [y],\dots,[x^{n-1}y]\right)&\text{ if } n\le m,\\
    \Lambda\left([x],\dots, [x^m]; [x^{n-m}y],\dots, [x^{n-1}y]\right)&\text{ if } n>m,
\end{align*}
with the induced differential. In either case, the differential $d$ takes generators of the form $[x^k]$ to zero and generators of the form $[x^k y]$ to elements in the subalgebra generated by the $[x^k]$'s, which must be sums of decomposables for degree reasons. So the linearized differential vanishes on indecomposables, and we obtain
\[
\pi_*\left(\SP^n(\CC\PP^m)\right)\otimes \QQ\cong
\begin{cases}
\QQ &\text{if } *=2,\dots,2\min\{m,n\};\\
\QQ &\text{if } *=2\max\{m,n\}+1,\dots,2n+2m-1;\\
0 &\text{otherwise}.
\end{cases}
\]

We sketch how to obtain the rational cohomology of $\SP^n(\CC\PP^m)$ from our model. Let $$A_j=\Lambda\left([x],[x^2],\dots, [x^n]; [y],[xy],\dots, [x^{j-1}y]\right)\quad (0\le j\le n)$$ with the restricted differential be an increasing sequence of sub-cdga's of $\SP^n(A)$. Here $A_0=\left(\Lambda\left([x],[x^2],\dots, [x^n]\right),d=0\right)$ and $A_n=\SP^n(A)$. Using the short exact sequence of cochain complexes
\[
0\to A_j\to A_{j+1}\to A_{j+1}/A_j\cong [x^{j}y]A_j\to 0
\]
and that $[x^{m+1}],\dots,[x^{m+n}]$ is a regular sequence in the algebra $\Lambda\left([x],\dots,[x^n]\right)$ \cite[Prop. 2.9]{CKW}, one can inductively show that the cohomology of $A_j$ is isomorphic to
\[
\Lambda\left([x],\dots,[x^n]\right)/\left([x^{m+1}],\dots,[x^{m+j}]\right).
\] 
In particular, if $m=1$ then the cohomology of $\SP^n(\CC\PP^1)$ is isomorphic to $$\Lambda\left([x],\dots,[x^n]\right)/\left([x^2],\dots,[x^{n+1}]\right).$$ From Newton's identities, one sees that $[x^{n+1}]=\pm\frac{1}{n!}[x]^{n+1}$ (modulo $[x^2],\dots,[x^n]$), so
$$\Lambda\left([x],\dots,[x^n]\right)/\left([x^2],\dots,[x^{n+1}]\right)\cong \Lambda\left([x]\right)/\left([x]^{n+1}\right).$$
This is as expected, since $\SP^n(\CC\PP^1)$ is homeomorphic to $\CC\PP^n$. Algebraically, this corresponds to $\SP^n\left(\QQ[x]/(x^2)\right) \cong \QQ[x]/(x^{n+1})$. 
\end{example}

\begin{example}
(See also \cite[p. 5]{FT10}.) Take the model $\left(\Lambda(x,y), dy=x^2\right)$ for $S^{2m}$ where $\deg x=2m$ and $\deg y=4m-1$. Then similarly to before, a minimal model for $\SP^n(S^{2m})$ is
\[
\Lambda\left([x],[x^{n-1}y]\right) \text{ with } d[x]=0\text{ and } d[x^{n-1}y]=\pm \frac{1}{n!}[x]^{n+1},
\]
which is isomorphic to $(\Lambda(x, z), d)$ with $dx = 0$ and $dz = x^{n+1}$. 

Taking the model $\left(\Lambda(y), dy=0\right)$ for $S^{2m-1}$ where $\deg y=2m-1$, we see that a minimal model for $\SP^n(S^{2m-1})$ is $\left(\Lambda([y]\right), d[y]=0)$. This means that $\SP^n(S^{2m-1})$ has the rational homotopy type of $S^{2m-1}$.
\end{example}

\end{document}